\documentclass[11pt]{amsart}
\usepackage{amsmath,amsthm,amssymb}
\usepackage{hyperref}
\usepackage{todonotes}

\newtheorem{theorem}{Theorem}[section]
\newtheorem*{theorem*}{Theorem}

\newtheorem*{lemma*}{Lemma}

\newtheorem*{proposition*}{Proposition}
\newtheorem{corollary}[theorem]{Corollary}
\newtheorem*{corollary*}{Corollary}

\theoremstyle{definition}
\newtheorem{definition}[theorem]{Definition}
\newtheorem*{definition*}{Definition}
\newtheorem{example}[theorem]{Example}
\newtheorem*{example*}{Example}

\newtheorem*{ques*}{Question}

\newtheorem*{claim*}{Claim}

\newtheorem*{remark*}{Remark}

\newcommand{\ii}{\item}

\newcommand{\surjto}{\twoheadrightarrow}

\newcommand{\defeq}{:=}

\newcommand{\RR}{\mathbb R}
\newcommand{\ZZ}{\mathbb Z}
\newcommand{\ZN}{\mathbb{Z}_{\ge 0}}
\newcommand{\Zc}[1]{\ZZ/#1}

\newcommand{\eps}{\varepsilon}

\newcommand{\pb}{p_{\mathrm{bas}}}

\usepackage{amsaddr}

\title{Linear polychromatic colorings of hypercube faces}
\date{\today}
\author{Evan Chen}
\address{Department of Mathematics, Massachusetts Institute of Technology}
\email{evanchen@mit.edu}


\begin{document}

\begin{abstract}
	A coloring of the $\ell$-dimensional faces of $Q_n$
	is called $d$-polychromatic if every embedded $Q_d$
	has every color on at least one face.
	Denote by $p^\ell(d)$ the maximum number of colors such that
	any $Q_n$ can be colored in this way.
	We provide a new lower bound on $p^\ell(d)$ for $\ell > 1$.
\end{abstract}

\maketitle

\section{Introduction}
Denote by $Q_n$ the $n$-dimensional hypercube on $2^n$ vertices.

\begin{definition}
	For $\ell \ge 0$, a \emph{$Q_\ell$-coloring} of $Q_n$ is a coloring of each of the
	the $\ell$-dimensional faces of $Q_n$ with one of $r \ge 1$ colors.
	For $d \ge \ell$, such a coloring is called
	\emph{$d$-polychromatic} if every embedded $Q_d$ contains all $r$ colors.

	For $d \ge \ell \ge 1$, we denote by $p^\ell(d)$ the maximum $r$
	for which a $d$-polychromatic $Q_\ell$-coloring is possible
	on every hypercube $Q_n$, for all $n \ge d$.
\end{definition}

The case $\ell = 1$ was first introduced in 2007 by
Alon, Krech, and Szab\'o in \cite{ref:alon}.
They prove the following result.
\begin{theorem*}
	[{\cite[Theorem 4]{ref:alon}}]
	For any $d \ge 1$,
	\[ \binom{d+1}{2} \ge p^1(d)
		\ge \left\lfloor \frac{(d+1)^2}{4} \right\rfloor. \]
\end{theorem*}
The lower bound is done through a construction which
in this paper will be called the \emph{basic construction},
described in Section~\ref{sec:simple}.
It was then shown by Offner in 2008 that in fact this construction is sharp.
\begin{theorem*}
	[\cite{ref:offner}]
	For any $d \ge 1$, we have
	\[ p^1(d) = \left\lfloor \frac{(d+1)^2}{4} \right\rfloor. \]
\end{theorem*}

Alon, Krech, and Szab\'o also suggest in \cite{ref:alon}
the problem of examining $p^\ell(d)$.
In 2015, Ozkahya and Stanton \cite{ref:basic}
gave a direct generalization of the basic construction
to prove the following.
\begin{theorem*}
	[\cite{ref:basic}]
	For any $d, \ell \ge 1$, let $0 < r \le \ell+1$
	be such that $r \equiv d+1 \pmod{\ell+1}$.
	Then
	\[ 
		\binom{d+1}{\ell+1}
		\ge p^\ell(d) 
		\ge 
		\left\lceil \frac{d+1}{\ell+1} \right\rceil^r
		\left\lfloor \frac{d+1}{\ell+1} \right\rfloor^{\ell+1-r}.
	\]
\end{theorem*}
Henceforth, denote the right-hand side as
\[
	\pb^\ell(d) \defeq
	\left\lceil \frac{d+1}{\ell+1} \right\rceil^r
	\left\lfloor \frac{d+1}{\ell+1} \right\rfloor^{\ell+1-r}
\]
for brevity.
For $\ell=1$ this coincides with the result of \cite{ref:alon}.
It is then natural to wonder whether an analog of Offner's result
holds for $\ell > 1$.
In a few small cases it was recently shown this is not the case;
Goldwasser, Lidicky, Martin, Offner, Talbot and Young prove in
\cite{ref:core} the following result.
\begin{theorem*}
	[{\cite[Theorems 20 and 21]{ref:core}}]
	We have $p^2(3) = 3$ and $p^2(4) \ge 5$.
\end{theorem*}
In contrast $\pb^2(3) = 2$ and $\pb^2(4) = 4$.

\bigskip

In the present paper we show the following more general result.
\begin{theorem}
	\label{thm:main}
	For $d \ge 4$, we have
	\[
		p^2(d) \ge
		\begin{cases}
			(k^2+1)(k+1) & d = 3k \\
			(k^2+k+1)(k+1) & d = 3k+1 \\
			(k^2+k+1)(k+2) & d = 3k+2.
		\end{cases}
	\]
	In particular, 
	\[ p^2(d) > \pb^2(d). \]
\end{theorem}

Our construction is by a so-called \emph{linear coloring},
defined at the end of Section~\ref{sec:simple}.
For concreteness, Table~\ref{tab:small} lists the values
of the construction for $4 \le d \le 12$,
as well as the bounds given by $\pb^2(d)$ and $\binom{d+1}{3}$.

\begin{table}[ht]
	\[
	\renewcommand{\arraystretch}{1.4}
	\newcommand{\ph}{\phantom{1}}
	\arraycolsep=12pt
	\begin{array}{rrrr}
		\multicolumn{1}{c}{d} & \multicolumn{1}{c}{\pb^2(d)} &
			\multicolumn{1}{c}{\text{Thm.~\ref{thm:main}}} &
			\multicolumn{1}{c}{\binom{d+1}{3}} \\ \hline
		d=\ph4 & 4=2\cdot2\cdot1 & 6=\ph3\cdot2 & 10 \\
		d=\ph5 & 8=2\cdot2\cdot2 & 9=\ph3\cdot3 & 20 \\
		d=\ph6 & 12=3\cdot2\cdot2 & 15=\ph5\cdot3 & 35 \\
		d=\ph7 & 18=3\cdot3\cdot2 & 20=\ph5\cdot4 & 56 \\
		d=\ph8 & 27=3\cdot3\cdot3 & 28=\ph7\cdot4 & 84 \\
		d=\ph9 & 36=4\cdot3\cdot3 & 40=10\cdot4 & 120 \\
		d=10 & 48=4\cdot4\cdot3 & 52=13\cdot4 & 165 \\
		d=11 & 64=4\cdot4\cdot4 & 65=13\cdot5 & 220 \\
		d=12 & 80=5\cdot4\cdot4 & 85=17\cdot5 & 286
	\end{array}
	\]
	\caption{For $4 \le d \le 12$, the values of $\pb^2(d)$, $\binom{d+1}{3}$,
	and the construction provided by Theorem~\ref{thm:main}.}
	\label{tab:small}
\end{table}

This easily implies that unlike $\ell = 1$,
we have $p^\ell(d) > \pb^\ell(d)$ for any $d > \ell > 1$
(with the $d=\ell+1$ case following from $p^2(3)=3$).
We state this formally as the following corollary.
\begin{corollary}
	\label{cor:limsup}
	For any $\ell > 1$ we have
	\[ \limsup_{d \to \infty} \left( p^\ell(d) - \pb^\ell(d) \right) = \infty. \]
\end{corollary}

The rest of the paper is structured as follows.
In Section~\ref{sec:simple} we present the background theory
and information for the problem,
and in Section~\ref{sec:family} we prove the construction
which gives the bound in Theorem~\ref{thm:main}.
Finally in Section~\ref{sec:upper} we mention some
upper bounds on the number of colors possible in a linear
polychromatic coloring.

\subsection*{Acknowledgments}
This research was funded by NSF grant 1358659 and NSA grant H98230-16-1-0026
as part of the 2016 Duluth Research Experience for Undergraduates (REU).

The author thanks Joe Gallian for supervising the research
and for suggesting the problem, as well as helpful comments
on early drafts of the paper.
The author would also like to acknowledge the anonymous referee
for several corrections and suggestions on the paper.

\section{Simple and linear colorings}
\label{sec:simple}
It is conventional to refer to the vertices of $Q_n$
with $n$-dimensional binary strings, and to represent
an embedded $Q_k$ by writing $\ast$ in the corresponding coordinates.
For example, in $Q_8$ the embedded $Q_2$ whose four vertices are
$01000011$, $01001011$, $01100011$, $01101011$,
is typically represented by
\[ 01 {\ast} 0 {\ast} 011. \]
We say that a $Q_\ell$-coloring is \emph{simple}
if the color of each $Q_d$ depends only on the number of $1$'s
in the $d+1$ regions (possibly empty) delimited by the $\ast$'s.
For example, in a simple $2$-polychromatic coloring of $Q_7$, the faces
$01{\ast}0{\ast}011$ and $10{\ast}0{\ast}101$ would be assigned the same color.

The following generalization of \cite[Claim 10]{ref:alon}
(present also as \cite[Lemma 18]{ref:core} and \cite[Claim 6]{ref:basic})
shows that in fact it suffices to only consider simple colorings.
The proof is a nice application of the Ramsey theorem.
\begin{theorem}
	Let $d \ge \ell \ge 1$ and assume $r \le p^\ell(d)$.
	Then for every $n \ge d$, there is a \emph{simple}
	$d$-polychromatic $Q_\ell$-coloring of $Q_n$ with $r$ colors.
\end{theorem}
Thus for the purposes of coloring, we can consider an embedded $Q_k$ in $Q_n$
as a sequence of nonnegative integers $(a_0, a_1, \dots, a_k)$
such that $a_i$ denotes the number of $1$'s
between the $i$th and $(i+1)$st star.
For example, $01{\ast}0{\ast}011$ can be identified with $(1,0,2)$.
In light of this a $Q_\ell$-coloring with colors from a set $S$
can be thought of as a function
\[ \chi : \ZN^{\ell+1} \surjto S. \]

We can now motivate the so-called basic colorings as follows.
\begin{definition}
	For $n \ge d \ge \ell \ge 1$,
	choose positive integers $m_0, m_1, \dots, m_\ell$ with sum $d+1$
	and consider the coloring
	\[ \chi : \ZN^{\ell+1} \surjto \bigoplus_{i=0}^\ell \Zc{m_i} \]
	by projection.
	This induces a $Q_\ell$-coloring of every $Q_n$
	with $m_0m_1 \dots m_\ell$ colors.

	We call any coloring of this form a \emph{basic} coloring.
\end{definition}
\begin{example}
	Let $d = 14$, $\ell = 2$, $m_0 = m_1 = m_2 = 5$.
	We claim this gives a basic $14$-polychromatic $Q_2$-coloring
	\[ \chi : \ZN^3 \surjto \Zc5 \oplus \Zc5 \oplus \Zc5 \]
	with $5^3=125$ colors.

	Consider an embedded $Q_{14}$ in some $Q_n$,
	which can be thought of as a sequence of $14$ stars.
	Select the $5$th and $9$th star as follows,
	and denote the remaining bits by $\eps_1, \dots, \eps_{12}$,
	as shown below.
	\[
		\begin{array}{cccc c cccc c cccc}
			\ast&\ast&\ast&\ast& \boxed{\ast}&
			\ast&\ast&\ast&\ast& \boxed{\ast}&
			\ast&\ast&\ast&\ast \\
			\eps_1 & \eps_2 & \eps_3 & \eps_4 &&
			\eps_5 & \eps_6 & \eps_7 & \eps_8 &&
			\eps_9 & \eps_{10} & \eps_{11} & \eps_{12}
		\end{array}
	\]
	This gives $2^{12} = 4096$ choices of $Q_2$ faces
	in our embedded $Q_{14}$.
	We claim that all colors are present
	among just these faces.

	Let $x$, $y$, $z$ denote the number of $1$'s 
	from the ambient $Q_n$ present in
	the three regions cut out by the boxed stars.
	Then, we wish to show that
	\[
		\chi\left(
		x + \eps_{1} + \dots + \eps_{4}, \;
		y + \eps_{5} + \dots + \eps_{8}, \;
		z + \eps_{9} + \dots + \eps_{12}
		\right)
	\]
	achieves all colors, which is obvious since
	$\eps_i + \eps_{i+1} + \eps_{i+2} + \eps_{i+3}$
	takes all possible values modulo $5$.
\end{example}
More generally, as shown in \cite[Theorem 1]{ref:basic},
every basic coloring is indeed seen to be $d$-polychromatic.
The lower bound $\pb^\ell(d)$ now follows by taking
the $m_i$ such that $|m_i-m_j| \le 1$ for all $1 \le i < j \le \ell$.

\begin{definition}
	More generally, a \emph{linear coloring} is one
	where the colors are selected from some (finite) abelian group $Z$,
	and which is induced by an additive map
	\[ \chi : \ZN^{\ell+1} \surjto Z. \]
\end{definition}

\section{A family of linear colorings}
\label{sec:family}

We now exhibit a family of linear
$d$-polychromatic $Q_\ell$-colorings.

\begin{theorem}
	\label{thm:crux}
	Let $n > t$ be positive integers.
	Set either
	\begin{itemize}
		\ii $m = t^2+1$ and $d = 2t+n-1$, where $t \ge 2$, or
		\ii $m = t^2+t+1$ and $d = 2t+n$, where $t \ge 1$.
	\end{itemize}
	Then the coloring
	\[
		\chi : \ZN^3 \surjto \Zc m \oplus \Zc n
		\quad\text{ given by }\quad
		(p,q,r) \mapsto (p-tq, p+q+r)
	\]
	is a linear $d$-polychromatic $Q_2$-coloring with $mn$ colors.
\end{theorem}

\begin{proof}
	We begin by addressing the first case $m = t^2+1$, $d = 2t+n-1$.
	Let $Z = \ZZ / m \oplus \ZZ / n$.

	Fix an embedding $Q_d$,
	which as usual we think of as a sequence of $d$ stars
	embedded in an ambient string of $1$'s and $0$'s.
	We can represent this with the diagram
	\[ x_0 \quad 
		\underbrace{\ast \quad x_1 \quad \ast \quad \cdots \quad %
		\ast \quad x_{d-1} \quad \ast}_{d\ \text{stars}} \quad x_d \]
	where $x_i$ denotes the number of $1$'s in the region
	delimited by those two stars.

	First, consider the family of squares cut out by the star pattern
	\[ \ast^{t-1} \quad \boxed\ast \quad \ast^{t-1}
		\quad \boxed\ast \quad \ast^{n-1} \]
	where we consider the squares formed when all the bits other
	than the $t$th and $2t$th bit are assigned a particular value.
	For example, the square
	\[
		\underbrace{0\cdots0}_{\text{$t-1$ $0$'s}}
		\quad \boxed\ast \quad
		\underbrace{0\cdots0}_{\text{$t-1$ $0$'s}}
		\quad \boxed\ast \quad
		\underbrace{0\cdots0}_{\text{$n-1$ $0$'s}}
	\]
	is assigned color $(X, S) \in Z$
	where $X = (x_0+\dots+x_{t-1})-t(x_{t}+\dots+x_{2t-1}) \pmod m$
	and $S = x_0 + \dots + x_d \pmod n$.

	Now suppose we vary the choice of assigned bits.
	First consider the last $n-1$ stars.
	Since $\{0, 1, \dots, n-1\}$ covers all residues modulo $n$,
	we see that the second coordinate is arbitrary,
	even regardless of the choices of the first $2(t-1)$ stars.
	Moreover, the first coordinate doesn't depend on the choice of
	these last $n-1$ stars.

	So we focus on the first coordinate.
	Let $0 \le u \le t-1$ and $0 \le v \le t-1$
	be the number of $1$'s we select in
	the first and second regions, respectively.
	(Thus the first coordinate receives color $X+u-tv$.)
	The values of $u-tv$ (modulo $m$) are given in the table
	\[
		\begin{array}{c|cccc}
			u-tv & u=0 & u=1 & \cdots & u=t-1 \\ \hline
			v=0 & 0 & 1 & \cdots & t-1 \\
			v=1 & t^2-t+1 & t^2-t+2 & \cdots & t^2 \\
			v=2 & t^2-2t+1 & t^2-2t+2 & \cdots & t^2-t \\
			\vdots & \vdots & \vdots & \ddots & \vdots \\
			v=t-1 & t+1 & t+2 & \cdots & 2t.
		\end{array}
	\]
	Thus, we see that we achieve exactly the colors
	with first coordinate in the set
	$X + \{ 0, 1, \dots, t-1, t+1, t+2, \dots, t^2 \}$
	so the colors not present are exactly those
	whose first coordinate is 
	\[ X + t \pmod m. \]

	Next, consider the family
	\[ \ast^{t} \quad \boxed\ast \quad \ast^{t-2}
		\quad \boxed\ast \quad \ast^{n-1} \]
	and this time define
	$Y = (x_0 + \dots + x_{t}) - t(x_{t+1} + \dots + x_{2t-1}) \pmod n$,
	which is the first coordinate of the analogous all-zero color.
	Again, consider varying the choice of assigned bits,
	this time with $u \in \{0, 1, \dots, t\}$ and $v \in \{0, \dots, t-2\}$.
	The values of $u-tv$ are given in the table
	\[
		\begin{array}{c|ccccc}
			u-tv & u=0 & u=1 & \cdots & u=t-1 & u=t \\ \hline
			v=0 & 0 & 1 & \cdots & t-1 & t \\
			v=1 & t^2-t+1 & t^2-t+2 & \cdots & t^2 & t^2+1 \\
			v=2 & t^2-2t+1 & t^2-2t+2 & \cdots & t^2-t & t^2-t+1 \\
			\vdots & \vdots & \vdots & \ddots & \vdots & \vdots \\
			v=t-2 & 2t+1 & 2t+2 & \cdots & 3t & 3t+1.
		\end{array}
	\]
	So by the same argument as in the previous case,
	the colors not present are exactly those whose
	first coordinate is in the set 
	\[ Y + \{t+1, t+2, \dots, 2t\} \pmod m. \]

	If $Y-X \notin \{1, 2, \dots, t\}$ then we are now done.
	Let $\delta = Y - X$ and henceforth assume $Y-X \in \{1, 2, \dots, t\}$.
	We denote by $k = X+t = Y+t+\delta$,
	and call any color of the form $(k, \bullet)$ a ``critical color.''
	We wish to show all $n$ critical colors are present on some other face.

	We consider the two families
	\[ \ast^{t-1} \quad \boxed\ast \quad \ast^{t}
		\quad \boxed\ast \quad \ast^{n-2} \]
	\[ \ast^{t} \quad \boxed\ast \quad \ast^{t-1}
		\quad \boxed\ast \quad \ast^{n-2} \]
	which we will call the ``first'' family and the ``second'' family.
	Let $C = -tx_{2t} \pmod m$.
	As before, the all-zero squares in these families
	receive the colors $(X+C, S) \in Z$ and $(Y+C, S) \in Z$, respectively.

	Define $u$ and $v$ as before and now let $0 \le w \le n-2$
	denote the number of $1$'s in the rightmost region.
	Again, we can exhibit two tables for $u$ and $v$ defined as before:
	for the first family we obtain a table
	\[
		\begin{array}{c|cccc}
			u-tv & u=0 & u=1 & \cdots & u=t-1 \\ \hline
			v=0 & 0 & 1 & \cdots & t-1 \\
			v=1 & t^2-t+1 & t^2-t+2 & \cdots & t^2 \\
			v=2 & t^2-2t+1 & t^2-2t+2 & \cdots & t^2-t \\
			\vdots & \vdots & \vdots & \ddots & \vdots \\
			v=t-1 & t+1 & t+2 & \cdots & 2t \\
			v=t & 1 & 2 & \cdots & t.
		\end{array}
	\]
	and for the second family we obtain a table
	\[
		\begin{array}{c|ccccc}
			u-tv & u=0 & u=1 & \cdots & u=t-1 & u=t \\ \hline
			v=0 & 0 & 1 & \cdots & t-1 & t \\
			v=1 & t^2-t+1 & t^2-t+2 & \cdots & t^2 & t^2+1 \\
			v=2 & t^2-2t+1 & t^2-2t+2 & \cdots & t^2-t & t^2-t+1 \\
			\vdots & \vdots & \vdots & \ddots & \vdots & \vdots \\
			v=t-2 & 2t+1 & 2t+2 & \cdots & 3t & 3t+1 \\
			v=t-1 & t+1 & t+2 & \cdots & 2t & 2t+1.
		\end{array}
	\]
	Note that every possible first coordinate is represented
	in both tables.

	Set $h = k-(X+C)$.
	Then the entries equal to $h$ in the first table
	correspond to choices $(u_1,v_1)$ which yield squares
	of critical color (regardless of the choice of $w$).
	In fact, as we vary $w$ the critical colors
	which are obtained are $(k, u_1+v_1+S+w)$,
	which is exactly the sequence of colors
	\[ (k, u_1+v_1+S), \; (k,u_1+v_1+S+1), \; \dots,
		\; (k, u_1+v_1+S+n-2). \]
	Thus the only critical color not present is $(k, u_1+v_1+S-1)$.

	Similarly, the entries equal to $h+\delta$ in the second table
	correspond to choices $(u_2,v_2)$ which yield squares in
	the second family with color $(k,\bullet)$
	(again regardless of the choice of $w$).
	For such a choice of $(u_2, v_2)$,
	by the same logic, the only critical color not present
	is $(k, u_2+v_2+S-1)$.

	So the problem reduces to the following.
	For arbitrary $h$ and $1 \le \delta \le t$,
	we need to show there exist
	$0 \le u_1 \le t-1$, $0 \le v_1 \le t$,
	$0 \le u_2 \le t$, and $0 \le v_2 \le t-1$
	so that
	\begin{align}
		u_1-tv_1 &\equiv h \pmod m \label{eq:first} \\
		u_2-tv_2 &\equiv h+\delta \pmod m \label{eq:second} \\
		u_1+v_1 &\not\equiv u_2+v_2 \pmod n. \label{eq:compat}
	\end{align}

	Intuitively, one can see this geometrically from the earlier tables.
	The quantities $u_i + v_i \pmod n$ correspond to ``northeast diagonals''
	in the table, which are ``spaced apart'' (since $n>t$)
	in such a way that a perturbation by $\delta < t$
	must move any $h$ into a different diagonal.

	We formalize this intuition in the following calculations.
	
	\begin{itemize}
		\ii In the case $h=t$, take
		$(u_1, v_1) = (t-1, t)$ and $(u_2, v_2) = (\delta - 1, t-1)$.
		Then $(u_1+v_1)-(u_2+v_2) = t+1-\delta$,
		which is not divisible by $n$ since $n > t$ and $1 \le \delta \le t$.

		\ii In the case $h=t-\delta$, take $(u_1, v_1) = (t-\delta, 0)$
		and $(u_2, v_2) = (t, 0)$.
		Then $(u_1+v_1)-(u_2+v_2) = -\delta$,
		again not divisible by $n$.

		\ii Now assume neither $h$ nor $h+\delta$ is equal to $t$.
		Then we can pick $(u_1, v_1)$ and $(u_2, v_2)$ satisfying
		\eqref{eq:first} and \eqref{eq:second},
		and actually $u_1, v_1, u_2, v_2 \in \{0, 1, \dots, t-1\}$.
		Let $A = u_2-u_1$ and $B = v_1-v_2$, so $A,B \in [-(t-1), t-1]$
		Now, subtracting \eqref{eq:first} from \eqref{eq:second} gives
		\[ A + tB \equiv \delta \pmod m. \]
		We have on one hand that $A+tB \le t-1 + t(t-1) < m < m+\delta$.
		On the other hand if $B \neq -(t-1)$ we also have
		$A+tB \ge -(t-1) + t(-t+2) \ge -t^2+t+1 > -m+\delta$.
		So there are only two possibilities:
		either
		\[ (A,B) = (\delta,0) \quad\text{or}\quad
		(A,B) = (-1-t+\delta, -(t-1)). \]
		In both cases, $A \neq B$ and $|A-B| \le \delta \le t < n$,
		hence \[ A \not\equiv B \pmod n \]
		which yields \eqref{eq:compat}.
	\end{itemize}
	Having completed all cases, this completes the proof of
	the situation $m = t^2+1$, $d = 2t+n-1$.

	The case where $m = t^2+t+1$ and $d = 2t+n$ is virtually identical,
	and so we will merely give a brief overview.
	The idea this time is to consider first the two families
	\[ \ast^{t-1} \quad \boxed\ast \quad \ast^{t}
		\quad \boxed\ast \quad \ast^{n-1} \]
	\[ \ast^{t} \quad \boxed\ast \quad \ast^{t-1}
		\quad \boxed\ast \quad \ast^{n-1} \]
	in order to once again reduce to a set of $n$ missing colors.
	Then one considers the family
	\[ \ast^{t-1} \quad \boxed\ast \quad \ast^{t+1}
		\quad \boxed\ast \quad \ast^{n-2} \]
	\[ \ast^{t} \quad \boxed\ast \quad \ast^{t}
		\quad \boxed\ast \quad \ast^{n-2} \]
	in the same manner as before.
\end{proof}

\begin{proof}
	[Proof of Theorem~\ref{thm:main}]
	In Theorem~\ref{thm:crux},
	take the following choices of parameters:
	\begin{itemize}
		\ii If $d=3k$, take $t=k$, $m=t^2+1$, $n=k+1$.
		\ii If $d=3k+1$, take $t=k$, $m=t^2+t+1$, $n=k+1$.
		\ii If $d=3k+2$, take $t=k$, $m=t^2+t+1$, $n=k+2$. \qedhere
	\end{itemize}
\end{proof}

\begin{proof}
	[Proof of Corollary~\ref{cor:limsup}]
	The result is immediate by Theorem~\ref{thm:main} for $\ell = 2$.

	For any general $\ell > 2$, let $d+1 = m_0 + m_1 + \dots + m_\ell$
	where $m_i \in \ZZ$ and $|m_i-m_j| \le 1$ for any $i$ and $j$.
	Let
	\[ \chi_0 : \ZZ_{\ge 0}^{\ell-2} \surjto \bigoplus_{j=0}^{\ell-3} \ZZ/m_j \]
	denote the basic coloring on $m_0 + \dots + m_{\ell-3} - 1$ stars,
	and let
	\[ \chi_1 : \ZZ_{\ge 0}^3 \surjto Z \]
	denote the coloring in Theorem~\ref{thm:main} on
	$m_{\ell-2} + m_{\ell-1} + m_\ell - 1$ stars.

	Then we can consider a coloring 
	\[
		\chi : \ZZ_{\ge 0}^{\ell+1}
		\surjto \left( \bigoplus_{j=0}^{\ell-3} \ZZ/m_j \right)
		\oplus Z
	\]
	defined by $\chi_0 \oplus \chi_1$,
	that applies $\chi_0$ to the first $\ell-2$
	components and $\chi_1$ on the last three.
	By construction $\chi$ also gives a $d$-polychromatic coloring,
	and the corollary follows.
\end{proof}

\begin{example}
	To illustrate Corollary~\ref{cor:limsup},
	suppose $d = 12$ and $\ell = 4$.
	Pick $m_0 = m_1 = m_2 = 3$ and $m_3 = m_4 = 2$;
	then the coloring $\chi_0$ has $\pb^1(5) = 9$ colors,
	while $\chi_1$ has $15$ colors
	(as in Theorem~\ref{thm:main}),
	and so the coloring $\chi_0 \oplus \chi_1 = 9 \cdot 15 = 135$ colors.
	On the other hand, $\pb^4(12) = 3^3 \cdot 2^2 = 108$ colors,
	according to the Theorem from \cite{ref:basic}.
\end{example}

\section{Upper bounds}
\label{sec:upper}
We do not have at present any upper bound for $p^2(Q_d)$
other than the simple $\binom{d+1}{3}$ bound.
In this section we briefly mention an upper bound for
the number of colors in a \emph{linear} $d$-polychromatic coloring.

Specifically, we use the geometry of numbers to prove the following.
\begin{theorem}
	Let $\chi : \ZZ_{\ge 0}^3 \surjto Z$ be a linear $d$-polychromatic coloring.
	For $d$ sufficiently large, we have
	\[ \left\lvert Z \right\rvert < \frac{26}{27} \binom{d+1}{3}. \]
\end{theorem}
\begin{proof}
	Let $N = |Z|$.
	Extend $\chi$ to a map $\ZZ^3 \surjto Z$ of abelian groups.
	Then consider $\ZZ^3$ as a tetrahedral lattice $\Lambda_0$ in $\RR^3$.
	In this case, the kernel of $\chi$ is a lattice $\Lambda$
	of index $N$ in $\ZZ^3$.

	Let $n = d-2$.
	Now if we consider the coloring of $Q_d$ itself by $\chi$
	(or really any embedding of $Q_d$ into $Q_N$ with all ambient bits zero),
	we see that the colors present are precisely those $\chi(x,y,z)$
	where $x+y+z \le n$, $x,y,z \in \ZZ_{\ge 0}$.
	Thus we obtain a regular tetrahedron $T$ of side length $n$
	in which all colors are present.

	On the other hand suppose that $\Lambda$ contains
	a nonzero vector $v$ which fits inside a
	regular tetrahedron of side length $s > 0$.
	Therefore for any $p \in \ZZ^3$, $\chi$ assigns the same color
	to both $p$ and $p+v$.
	In particular, this implies all the colors are present
	in a frustum of $T$ with height $s$ layers;
	this gives a bound of
	\begin{equation}
		N \le \binom{d+1}{3} - \binom{d+1-s}{3}.
		\label{eq:tetrahedron}
	\end{equation}

	Now let $c$ be the length of the shortest nonzero vector in $\Lambda$.
	Then since a tetrahedron has height equal to $\sqrt{2/3}$ times 
	its side length, we may take
	\begin{equation}
		s = \left\lceil \sqrt{3/2} c \right\rceil.
		\label{eq:height}
	\end{equation}
	Next we bring in the theory of sphere packing.
	Observe that if we construct spheres of diameter $c$
	centered at each point in $\Lambda$,
	then we have obtained a packing of spheres in $\RR^3$.
	We have $\det(\Lambda) = N \det(\Lambda_0)$,
	but $\Lambda_0$ is known to be an optimal packing of $3$-spheres
	(see e.g.\ \cite{ref:spheres}), and so from this we deduce that
	\begin{equation}
		0 < c \le \sqrt[3]{N}.
		\label{eq:sphere}
	\end{equation}

	Collating \eqref{eq:tetrahedron}, \eqref{eq:height}, \eqref{eq:sphere}
	together we deduce the inequality
	\[ c \le \sqrt[3]{\binom{d+1}{3} - \binom{d-\sqrt{3/2} c}{3}}. \]
	Thus, we have
	\begin{align*}
		6c^3 &\le (d^3-d) +
		\left(\sqrt{3/2}c - d\right)
		\left(\sqrt{3/2}c - (d-1)\right)
		\left(\sqrt{3/2}c - (d-2)\right) \\
		&=\sqrt{27/8}c^3 - 9/2(d-1) c^2 + (3d^2-6d+2)\sqrt{3/2}c + (3d^2-3d).
	\end{align*}
	We can rewrite this as
	\[
		\left( 6-\sqrt{\frac{27}{8}} \right) \left( \frac cd \right)^3
		+ \frac{9}{2} \left( \frac cd \right)^2
		- 3\sqrt{\frac32} \left( \frac cd \right)
		\le O\left( \frac 1d \right).
	\]
	Solving the resulting quadratic,
	we see that for sufficiently large $d$ we have $c/d \le 0.5434$,
	and thus $s < 0.5434\sqrt{3/2}d + 1 < 0.666d$.
	Finally, using \eqref{eq:tetrahedron} we have
	\begin{align*}
		N &\le \binom{d+1}{3} - \binom{0.334d+1}{3} \\
		&< \left( 1 - \left( \frac 13 \right)^3 \right) \binom{d+1}{3} \\
		&= \frac{26}{27} \binom{d+1}{3}
	\end{align*}
	again for $d$ sufficiently large.
\end{proof}

It would be interesting if any stronger upper bounds could be proven for
polychromatic colorings, linear or otherwise.

\bibliographystyle{hplain}
\bibliography{refs-cube}

\end{document}